\documentclass[12pt,a4paper]{article}
\usepackage{amsmath}
\usepackage{amssymb}
\usepackage{accents}

\usepackage{mathtools}
\usepackage{amsthm}
\usepackage[usenames]{color}

\usepackage{graphicx}

\newcommand{\vr}{\varrho}

\newcommand{\eps}{\varepsilon}
\newcommand{\Z}{{\mathbb Z}}

\newcommand{\B}{{\mathsf B}}

\newcommand{\V}{{\mathcal V}}

\newcommand{\R}{{\mathbb R}}

\newcommand{\cCC}{{\mathcal C}}

\newcommand{\NN}{{\mathcal N}}

\newcommand{\Tsp}{{\mathcal T}}

\newcommand{\F}{{\mathcal F}}

\newcommand{\RI}{\mathop{\mathrm{RI}}}

\newcommand{\s}{{\widehat S}}

\newcommand{\I}{{\mathcal I}}
\newcommand{\J}{{\mathcal J}}

\let\phi=\varphi

\newcommand{\hL}{\widehat{L}}

\newcommand{\convlaw}{\stackrel{\text{\tiny law}}{\longrightarrow}}

\newcommand{\IP}{{\mathbb P}}
\newcommand{\IE}{{\mathbb E}}

\newcommand{\distTV}{\mathop{dist}\nolimits_{\mathit{TV}}}

\DeclareMathSymbol{\widehatsym}{\mathord}{largesymbols}{"62}

\newcommand{\tZ}{\widetilde{Z}}

\newcommand{\capa}{\mathop{\mathrm{cap}}}

\newcommand{\hm}{\mathop{\mathrm{hm}}\nolimits}
\newcommand{\hhm}{\mathop{\widehat{\mathrm{hm}}}\nolimits}

\newcommand{\htau}{\widehat{\tau}}

\newcommand{\ttL}{{\widetilde L}}

\allowdisplaybreaks

\newtheorem{theo}{Theorem}
\newtheorem{lem}[theo]{Lemma}

\newtheorem{prop}[theo]{Proposition}

\newtheorem{rem}[theo]{Remark}

\title{Two-dimensional random interlacements: 
0-1 law and the vacant
set at criticality}
\author{Orph\'ee Collin$^{1}$ \and
 Serguei~Popov$^{2}$}

\begin{document}

\maketitle

{\footnotesize 
\noindent $^{~1}$Universit\'e Paris Cit\'e 
and Sorbonne Universit\'e, CNRS, Laboratoire de Probabilit\'es,
Statistique et Mod\'elisation, F--75013 Paris, France
\\
\noindent e-mail:
\texttt{collin@lpsm.paris}

\noindent $^{~2}$Centro de Matem\'atica, University of 
Porto, Porto, Portugal\\
\noindent e-mail: \texttt{serguei.popov@fc.up.pt}

}

\begin{abstract}
 We correct and streamline the proof of the fact
that, at the critical point $\alpha=1$, 
the vacant set of the two-dimensional random interlacements
is infinite~\cite{CP17}. Also, we prove a zero-one
law for a natural class of tail events related to
the random interlacements.
\end{abstract}

This note is about the model of 
random interlacements on the two-dimen\-sional
integer lattice~$\Z^2$.
We define and discuss this model in a more detailed
way later, but, informally, it is a ``Poissonian
soup'' of (double-infinite)
trajectories of simple random walk
\emph{conditioned} on never entering the origin
(we define these formally later);
$\alpha>0$ stands for the intensity 
parameter (or ``level'') of the corresponding
Poisson process of trajectories.
A site~$x\in\Z^2$
is called \emph{vacant} if no trajectory 
of that soup passes through
it. In~\cite{CPV16}, it was shown that $\alpha=1$
is \emph{critical}, in the following sense: 
if $\alpha<1$ then there are infinitely many vacant
sites a.s., while for $\alpha>1$ a.s.\ there is only 
a finite number of these. The question of
what happens exactly at the critical level
$\alpha=1$ was left open
in~\cite{CPV16} and was the main subject of
the subsequent paper~\cite{CP17}.
Let us restate here Theorem~1.2 of~\cite{CP17}:
\begin{theo}
\label{t_critical_vacant}
At the (critical) level $\alpha=1$, 
a.s.\ there are infinitely many vacant
sites.
\end{theo}
Unfortunately, as we explain below, the proof
of this result in~\cite{CP17} contains a flaw;
one of the main purposes of this note is 
to rectify that proof. While doing so, 
we also make it conceptually much simpler
by taking advantage of a certain 0-1 law valid
for two-dimensional random interlacements
(stated as Theorem~\ref{t_spacewise_01} below),
which can be seen as the other main result of
this note.

Let us now quickly recall the relevant notations and 
definitions
(we will mostly use the notations of~\cite{CP17};
see also Chapters~3, 4, and~6 of~\cite{P21}). 
In the following, 
SRW stands for simple 
random walk on~$\Z^2$
or on the torus $\Z^2_n:=\Z^2/n\Z^2$. We write $x\sim y$
when~$x$ and~$y$ are neighbours in~$\Z^2$ or~$\Z^2_n$.
Being $\|\cdot\|$ the Euclidean norm,
$\B(x,r)=\{y\in\Z^2: \|y-x\|\leq r\}$ is the (discrete)
disk of radius~$r$ centered at~$x$, and $\B(r):=\B(0,r)$
(we also sometimes consider such disks placed on tori).
Then, $|A|$ stands for the cardinality of~$A$,
$\partial A = \{x\in A : \text{there is }
y\in\Z^2\setminus A \text{ such that }x\sim y\}$
is the (inner) boundary of $A\subset \Z^2$,
$\hm_A(\cdot)$ is the harmonic measure (with respect to SRW)
on~$A$ (it is concentrated on~$\partial A$ and is 
only well-defined when~$A$ is finite).
For $A\subset A'$, an \emph{excursion}
between~$\partial A$ and~$\partial A'$ is a (finite)
piece of a nearest-neighbor trajectory that starts
at~$\partial A$ and ends on its first visit to~$\partial A'$;
in this paper, we will only consider excursions between the 
boundaries of concentric disks.

The conditioned SRW in two dimensions is defined as
the Doob's $h$-transform of the SRW with respect
to its potential kernel~$a(\cdot)$: 
for $0\neq x\sim y$, the transition probability
 from~$x$ to~$y$ is equal to~$\frac{a(y)}{4a(x)}$.
It is possible to show (see~\cite{CPV16} or
Chapter~4 of~\cite{P21}) that this new 
random walk is transient and reversible.
Then, as explained in~\cite{CPV16}
(see also Chapter~6.2 of~\cite{P21}), 
one can define the random interlacements
canonically using the results of~\cite{Tei09}
(we also mention that a somewhat different
approach was used in~\cite{Rod19}).
We will use the abbreviation~$\RI(\alpha)$
 for two-dimensional random interlacements
at level~$\alpha>0$, and $\V^\alpha$ will denote 
the vacant set (i.e., the set of vacant sites)
of $\RI(\alpha)$. It is also important to have in mind that,
differently from the ``classical'' random interlacements
introduced in~\cite{S10} in
 ``transient'' dimensions $d\geq 3$,
in two dimensions the process is \emph{not}
stationary in space (in particular, 
as shown in~\cite{CPV16}, the probability 
that~$x\in\Z^2\setminus\{0\}$
is vacant for~$\RI(\alpha)$
is of order~$\|x\|^{-\alpha}$).


As mentioned before,
the construction described
on the last pages of~\cite{CP17}, 
unfortunately, contains a flaw.
This was overlooked by the authors of~\cite{CP17} due to 
a mistake in the excursion count calculation
just before~(83) 
 (for the SRW on the torus,
before time~$t_k$ defined there):
the correct leading term should be 
$\frac{1}{2\ln \gamma}\ln^2 b_k$ instead of
$\frac{2}{\ln \gamma}\ln^2 b_k$.
That would ruin the subsequent argument since 
the number of excursions generated by the interlacements
would be more than enough to cover the corresponding disk.

Below, we present a corrected argument.
It is possible to modify the construction
of~\cite{CP17} ``mechanically'' to address that issue 
(basically, in \cite{CP17}'s notation,
considering $B_k=\B\big(v_k, \frac{b_k}{\ln^2 b_k}\big)$
instead of $B_k=\B(v_k, b_k^{1/2})$ and modifying~$B'_k$ 
and~$b_k$ accordingly would suffice), but we prefer to
use this paper to present
a substantially simpler and cleaner 
(as explained right after~\eqref{01_vacant})
way of proving the result. 

\begin{proof}[Proof of Theorem~\ref{t_critical_vacant}]
First, assume that we have the following general fact:
\begin{equation}
\label{01_vacant}
 \text{For any }\alpha>0, \quad 
  \IP\big[|\V^\alpha|=\infty\big] = 0 \text{ or }1.
\end{equation}
Then, with~\eqref{01_vacant} at hand, if one assumes
that $|\V^1|<\infty$ a.s., that would mean that
for any $\eps>0$ there is $R=R(\eps)$ such that
$\IP[\V^1\subset \B(R)]>1-\eps$. Therefore,
to obtain a contradiction, it suffices to prove that,
 for any fixed~$R>0$,
the vacant set~$\V^1$ contains a site outside 
of~$\B(R)$ with uniformly positive probability. 
This means that one does not need to control 
the dependencies in the whole sequence of events
$(\{\V^1\cap B_k\neq\emptyset\})_{k\geq 1}$
(as was done in~\cite{CP17}, see the proof of~(75))
to prove that an infinite number of these occurs
a.s.; just estimating (from below) the probability
of a generic event of that sequence would suffice.

We now derive such an estimate. 
To do this in a cleaner way (compared to~\cite{CP17}),
we need another preliminary result.
\begin{prop}
\label{p_ind_exc_notcover}
 Let $\gamma>1$ and $\beta>0$ be fixed numbers.
Consider $\frac{2\ln^2 n}{\ln \gamma}
 - (1+\beta)\frac{\ln n\ln\ln n}{\ln\gamma}$
 independent SRW's excursions between~$\partial\B(n)$
and~$\partial\B(\gamma n)$, with starting points
chosen according to~$\hm_{\B(n)}$. Then,
for some positive~$c'$ and~$c''$ depending only
on~$\beta$ and~$\gamma$
\begin{equation}
\label{eq_ind_exc_notcover}
\IP[\B(n)\text{ is not completely covered
by these excursions}] 
\geq 1 - c' \exp\big(-c''(\ln\ln n)^2\big)
\end{equation}
for all~$n\geq 3$.
\end{prop}

\begin{proof}
First, we give an outline of the proof:
Let us consider the torus~$\Z^2_{m}$ 
with $m=\lfloor 3\gamma n\rfloor$,
with concentric disks~$\B(n)$
and~$\B(\gamma n)$ embedded there in a natural way.
It is known that, if one runs the SRW on that
torus up to time $t_{m,\beta}:=\frac{4}{\pi}m^2\ln^2 m
- (1+\beta)\frac{2}{\pi}m^2\ln m \ln\ln m$,
the number of excursions
between~$\partial\B(n)$ and~$\partial\B(\gamma n)$
will be well concentrated 
around $\frac{2\ln^2 n}{\ln \gamma}
 - (1+\beta)\frac{\ln n\ln\ln n}{\ln\gamma}$. 
It is also known~\cite{A21} that
with high probability there will still be uncovered sites
on the torus by time~$t_{m,\beta}$,
and that it is possible to relate SRW's excursions
to independent excursions via e.g.\ soft local times
(as explained below); unfortunately, this is 
not yet enough to obtain~\eqref{eq_ind_exc_notcover}
since one cannot apriori exclude the possibility
that the uncovered sites are ``spatially concentrated''.
(This way, one can only obtain that the probability
in~\eqref{eq_ind_exc_notcover} is uniformly bounded
 from below, which was still enough for the argument
of~\cite{CP17} at the cost of some additional 
technical difficulties.)
 In fact, one can still obtain~\eqref{eq_ind_exc_notcover}
(actually, \eqref{uncov_B(n)} below, that would by its turn
imply~\eqref{eq_ind_exc_notcover})
by a suitable modification of the arguments of~\cite{A21}
(or, in the continuous setting, by a modification
of the arguments of~\cite{BK14})
as the authors were able to find out
thanks to private communications 
with Y.~Abe and O.~Zeitouni.
However, we prefer to present a ``softer'' way to 
obtain~\eqref{eq_ind_exc_notcover}, which only uses
the main cover time result and does not 
require plunging into the technical details
of~\cite{A21,BK14}.

Also, we have to mention that, in principle,
it should be possible to prove 
Proposition~\ref{p_ind_exc_notcover} in a ``direct''
way, i.e., without referring to the SRW on the torus
at all. However, it is quite likely that such a proof will
have a similar complexity level as the corresponding
(rather lengthy and technically involved) proof
of the torus covering result. Therefore, while 
it may indeed be convenient to eventually
have such a direct
proof available, we feel that this note would not 
be the right place for it.

We prove~\eqref{eq_ind_exc_notcover} in three steps:
first, as mentioned above, we use the known results
for the cover time of the torus to prove that 
the probability of the corresponding event for 
independent excursions is uniformly positive. 
Next, we use that uniform positivity to refine a bit
the torus covering result: we will show that a fixed 
disk of radius proportional to the size of the torus
will contain an uncovered site with high probability.
Then, we translate this ``improved'' result 
back to independent excursions.
It is also worth mentioning that the harmonic 
measure~$\hm_{\B(n)}$ is not exactly the 
``correct'' one for choosing the starting points
of the excursions. With $A\subset A'$ and
$y\in\partial A$, define
\begin{equation}
\label{df_hm_AA'}
 \hm_A^{A'}(y) = \IP_y[\tau_1(\partial A')<\tau_1(A)]
\Big(\sum_{z\in\partial A}\IP_z[\tau_1(\partial A')
<\tau_1(A)]\Big)^{-1},
\end{equation}
where $\tau_1(A)=\min\{k\geq 1: S_k\in A\}$ is the hitting
time of~$A$ by the SRW $(S_n, n\geq 0)$.
For example, for the excursion process 
(between~$\partial\B(n)$ and~$\partial\B(\gamma n)$)
generated by the SRW on the torus, it is
the measure~$\hm_{\B(n)}^{\B(\gamma n)}$ 
that will be invariant for it. 
It is, however, quite close to~$\hm_{\B(n)}$:
due to Lemma~2.5 of~\cite{CP17}, we have
\begin{equation}
\label{hm_close}
 \hm_{\B(n)}^{\B(\gamma n)}(y) 
  = \hm_{\B(n)}(y)\big(1+O(n^{-1})\big);
\end{equation}
since we only need to deal with $O(\ln^2 n)$
excursions, there will be no essential difference
(in particular, Proposition~\ref{p_ind_exc_notcover}
also holds with the starting points chosen 
according to~$\hm_{\B(n)}^{\B(\gamma n)}$).

\smallskip
\noindent
\textit{Step 1.} For $k\geq 3$, denote
 $\psi_{k,\beta}=\frac{2\ln^2 k}{\ln \gamma}
 - (1+\beta)\frac{\ln k\ln\ln k}{\ln\gamma}$.
It is important to keep in mind that this 
quantity does not change a lot when one changes
the value of~$k$: if $k'=k \times(\ln\ln k)^M$ where~$|M|$
is bounded from above by a universal constant,
we have 
$\psi_{k',\beta}=\psi_{k,\beta}+O(\ln k \ln\ln\ln k)$.
In the arguments below, when we consider 
 \emph{independent} excursions in the context of
 covering~$\B(k)$, we always assume 
 that these excursions are
 between~$\partial\B(k)$
 and~$\partial\B(\gamma k)$, and  
 the starting points of these 
are sampled from~$\hm_{\B(k)}$.
Let us first prove that for all large enough~$n$ 
we have
\begin{equation}
\label{c>0_ind_exc_notcover}
\IP[\B(n)\text{ is not completely covered
by $\psi_{n,\beta}$ independent excursions}] \geq c_\beta
\end{equation}
for some $c_\beta>0$. 
For this, we note that, on the torus $\Z^2_m$
(as above, with $m=\lfloor 3\gamma n\rfloor$),
the number of SRW's excursions 
between~$\partial\B(n)$
and~$\partial\B(\gamma n)$ 
up to time~$t_{m,\beta/3}$
will be at most $\psi_{n,\beta/2}$ with probability
at least $1-c_1\exp\big(-c_2 (\ln\ln n)^2\big)$
(this follows e.g.\ from Lemma~2.11 of~\cite{CP17},
take $\delta=\frac{\eps\ln\ln n}{\ln n}$ with 
small enough~$\eps>0$).
Next, we use soft local times~\cite{P21,SLT},
to construct a coupling of the SRW's excursions
with independent excursions.
More specifically, we use the construction
with marked Poisson process (with excursions as marks)
described in Section~2.2 of~\cite{CP17};
 let us recall the notation $L_k(\cdot)$ 
for the soft local time on~$\partial\B(n)$ generated
by~$k$ excursions of the SRW on the torus,
and let $\ttL_k(y)=({\tilde \xi}_1 + \cdots+{\tilde \xi}_k)
 \hm_{\B(n)}(y)$ be the soft local times 
 for the independent
excursions.\footnote{${\tilde \xi}$'s are i.i.d.\ 
Exponential(1) random variables, used in the definition
of the soft local time.}
Using~\eqref{hm_close}
and Lemma~2.9 of~\cite{CP17}
(with $\theta=\frac{\eps\ln\ln n}{\ln n}$, where~$\eps>0$ 
is small enough), we obtain 
\begin{equation}
\label{compare_SLTs_torus}
\IP\big[\ttL_{\psi_{n,\beta}}(y)\leq L_{\psi_{n,\beta/2}}(y)
\text{ for all }y\in\partial\B(n)\big]
  \geq 1-c'_1\exp\big(-c'_2 (\ln\ln n)^2\big),
\end{equation}
meaning that the set of independent excursions
will be contained in the set of the SRW's
excursions with high probability.
Theorem~1.1 of~\cite{A21} implies that at
time~$t_{m,\beta/3}$ there will be an uncovered site 
in the torus with high probability, so (since~$\B(n)$
occupies a constant fraction of the volume of the torus) there 
will be an uncovered site in~$\B(n)$ with at least
a constant probability. This shows~\eqref{c>0_ind_exc_notcover}.

\smallskip
\noindent
\textit{Step 2.}
Now, take a large 
$h>\gamma$. Let~$\kappa_h$
be the maximal number of nonoverlapping disks
of radius~$\gamma h^{-1}n$ and with centers at integer points
that fit inside~$\B(n)$;
clearly, $\kappa_h$ is of order~$h^2$.
Let $x_1,\ldots,x_{\kappa_h}$ be these centers,
and denote $B_j=\B(x_j,h^{-1}n)$,
$B'_j=\B(x_j,\gamma h^{-1}n)$, $j=1,\ldots,\kappa_h$.
For $j=1,\ldots,\kappa_h$,
let $\tZ^{(j), k}, k\geq 1$ be the independent
excursions between~$\partial B_j$ and~$\partial B'_j$
(again, with the initial points chosen according 
to~$\hm_{B_j}$).
Consider \emph{independent} events
\[
 U_j = \Big\{B_j\text{ is not fully covered by }
  \tZ^{(j), 1},\ldots, \tZ^{(j), \psi_{h^{-1}n,\beta/2}}\Big\}, 
\]
and let $U=U_1\cup\ldots \cup U_{\kappa_h}$;
then, \eqref{c>0_ind_exc_notcover} implies that
\begin{equation}
\label{notcover_atleastone}
 \IP[U] \geq 1-(1-c_{\beta/2})^{\kappa_h},
\end{equation}
which can be made arbitrarily close to~$1$ by the choice
of~$h$. But, similarly to~\eqref{compare_SLTs_torus}
of Step~1, 
we can argue (also using Lemma~2.10
of~\cite{CP17}) that with probability at least 
$1-c''_1 \kappa_h\exp\big(-c''_2 (\ln\ln n)^2\big)$
the soft local time of real SRW's excursions 
up to time~$t_{m,\beta}$
is below the soft local time of
the above independent excursions,
meaning that~$\B(n)$ contains an uncovered site
at time~$t_{m,\beta}$ with probability
at least 
$1-(1-c_{\beta/2})^{\kappa_h}
-c''_1 \kappa_h\exp\big(-c''_2 (\ln\ln 
n)^2\big)$.
One can then choose $h=\ln\ln n$
(so that $\kappa_h\asymp (\ln\ln n)^2$)
to obtain that, for any fixed $\beta>0$,
\begin{equation}
\label{uncov_B(n)}
 \IP[\B(n) \text{ is not fully covered
by SRW on $\Z^2_m$ at time }t_{m,\beta}]
 \to 1
 \quad \text{as } n\to\infty.
 \end{equation}

\smallskip
\noindent
\textit{Step 3.}
Now, to obtain~\eqref{eq_ind_exc_notcover},
we just repeat what was done at Step~1, 
but with~\eqref{uncov_B(n)} to hand
(instead of the main cover time result of~\cite{A21}).
This concludes the proof 
of Proposition~\ref{p_ind_exc_notcover}.
\end{proof}
In fact, as the reader will see, we will only need 
the above result with one fixed~$\beta>0$;
also, we will only need 
the probability in~\eqref{eq_ind_exc_notcover}
to converge to~$1$ as $n\to \infty$. 
Still, we decided to state
Proposition~\ref{p_ind_exc_notcover}
in a more general form for the sake of possible future
reference, as proving this more general version 
does not require any considerable extra effort anyway.

We continue proving Theorem~\ref{t_critical_vacant}.
Fix a large (integer) $s>0$ 
and a site $x_s\in\Z^2$ such that $\|x_s\|=s$ (for example,
$x_s=(s,0)$);
then, define $B=\B\big(x_s, \frac{s}{\ln^2 s}\big)$,
$B'=\B\big(x_s, e \frac{s}{\ln^2 s}\big)$
(we took $\gamma = e$ just to get rid of $\ln \gamma$ terms
in the formulas).
In the following, $\RI$ stands for $\RI(1)$, 
and we remind the reader that the~$\RI$'s trajectories
are \emph{conditioned} SRWs.
Then, Lemma~2.7 of~\cite{CP17} implies that
\begin{equation}
\label{cap_B}
\capa\big(\{0\}\cup B\big) 
 = \frac{2}{\pi} \ln s 
  \times \big(1+O\big(\tfrac{\ln\ln s}{\ln s}\big)\big),
\end{equation}
and Lemma~2.6 of~\cite{CP17} implies that,
for any $x\in\partial B'$
\begin{equation}
\label{hit_B_from_partial_B'}
 \IP_x[\htau(B)<\infty]
  = 1 - \frac{1}{\ln s}
           \big(1+O\big(\tfrac{\ln\ln s}{\ln s}\big)\big),
\end{equation}
where $\htau(B)$ is the hitting time by the conditioned SRW.
So, similarly to the argument in~\cite{CP17},
the number~$N$ of RI's excursions between~$\partial B$
and~$\partial B'$ is compound Poisson with rate
$\pi\capa\big(\{0\}\cup B\big)= 2\ln s 
  \times \big(1+O\big(\tfrac{\ln\ln s}{\ln s}\big)\big)$
and (approximately) exponentially distributed
summands of mean $\ln s \times
           \big(1+O\big(\tfrac{\ln\ln s}{\ln s}\big)\big)$.
Therefore, the expected number of these excursions
is $2\ln^2 s \times
           \big(1+O\big(\tfrac{\ln\ln s}{\ln s}\big)\big)$,
and, moreover, it is straightforward to argue that
\begin{equation}
\label{Nk_Normal}
\frac{N - 2\ln^2 s}{2\ln^{3/2}s}
 \convlaw \text{standard Normal.}
\end{equation}
The fact that typical deviations of the excursion count
are rather large (of order of the mean to power~$\frac{3}{4}$) 
plays the key role in the infiniteness of the critical
vacant set: it does not cost much to have a downward
fluctuation in the 
number of RI's excursions between~$\partial B$
and~$\partial B'$ which makes it ``subcritical''.
Note that, from the above discussion it follows
that
\begin{equation}
\label{prob_count_down}
 \IP\big[N\leq 2\ln^2 s - 
  \ln^{3/2}s\big] \geq \frac{1}{4}
\end{equation}
for all large enough~$s$. 
Abbreviate $m_0=\frac{s}{\ln^2 s}$.
Next, again by means
of soft local times,
we construct a coupling of the RI excursions
with (say) 
$\psi^*:=2\ln^2 m_0 - 3\ln m_0 \ln\ln m_0$
independent ones (generated by a SRW with starting
points chosen by $\hm_B$).
Let us denote by~$\hL_k(\cdot)$ the soft local time
generated by~$k$ RI's excursions (again,
see Section~2.2 of~\cite{CP17} for definitions).
First, observe that
due to Lemma~3.3~(ii) of~\cite{CPV16},
one can successfully couple one SRW's excursion
with one conditioned SRW's excursion
started at the same site
with probability at least~$1-O(\ln^{-3} s)$.
Since there are only $O(\ln^2 s)$ of these,
we see that it is possible to couple the ``marks''
(of the marked Poisson process) with high
probability; so, we only need to couple the initial points now.
Note that, due to Lemma~2.5 of~\cite{CP17},
\begin{align}
 \hhm_B^{B'}(y) &= \hm_B(y)\big(1+O(\ln^{-3} s)\big),
 \label{est_hhmBB'} 
\end{align}
for all $y\in\partial B$.
Using the above with Lemma~2.9 of~\cite{CP17}
(take $\theta=\ln^{-3/4}s$ there),
one obtains that 
\begin{equation}
\label{SLT_RI_independent}
 \IP\big[\hL_{2\ln^2 s-\ln^{3/2}s}(y)
  \leq L_{\psi^*}(y)\text{ for all }y\in\partial B\big]
     \geq 1-c\exp(-c'\ln^{1/2}s).
\end{equation}
Therefore, by Proposition~\ref{p_ind_exc_notcover} 
and~\eqref{prob_count_down},
 RI's excursions leave a vacant site in~$B$
with probability 
at least $\frac{1}{4}-c\exp(-c'\ln^{1/2}s)$.
The above shows that $\V^1\cap B\neq \emptyset$
with uniformly positive probability, as desired.
This concludes the proof of Theorem~\ref{t_critical_vacant}
(under the assumption that~\eqref{01_vacant} holds).
\end{proof}

\begin{rem}
To prove the above result, another possible route
would be using Theorem~2.6 of~\cite{BGP18} together
with Proposition~\ref{p_ind_exc_notcover}: since we know
that $2\ln^2 s - C\ln^{3/2}s$
independent excursions do not cover~$B$ with 
probability close to~$1$, the same number
of RI's excursions also will not do that at least 
with a constant probability.
\end{rem}

We are left with the task of proving~\eqref{01_vacant}.
The event $\{|\V^\alpha|=\infty\}$ \emph{looks like}
a tail event, in the sense that it is not affected by what
happens in any finite region. Our intuition then says
that a 0-1 law should hold for it; it is however not 
immediate to obtain such a law for the model of random
interlacements since even one trajectory can 
(and will) affect what 
happens in arbitrarily remote regions. 
In fact, one can still prove~\eqref{01_vacant}
in a direct way (basically, for this 
one needs to argue that a finite number of
conditioned SRW's trajectories
cannot cover almost all sites of a fixed infinite
subset of~$\Z^2$), but we prefer to prove a more
general 0-1 law, also for future reference.
We remark also that \cite[Section~2]{S10}
contains a version of 0-1 law for the ``classical''
random interlacements (i.e., in dimensions $d\geq 3$);
however, it makes use of the translational invariance
property which is absent in two dimensions.

We need to introduce more notations.
Consider the sequence 
$\Lambda_n=\{x\in\Z^2: n-1 < \|x\| \leq n\}$ of
disjoint and nonempty subsets of~$\Z^2$; we then have
$\Z^2\setminus\{0\} = \bigcup_{n\geq 1} \Lambda_n$.
Note also that if an unbounded nearest-neighbor trajectory
passes through~$\Lambda_{n_0}$, then it has to pass 
through all $\Lambda_n$ for $n\geq n_0$ (i.e., the trajectories
cannot overjump these subsets; this is because the norms of neighboring
sites cannot differ by more than one unit).
In addition, we denote $\Theta_n=\bigcup_{j\geq n} \Lambda_j$.
Being~$\I$ an interlacement configuration
(seen as a countable set of trajectories) and~$K$
a finite subset of $\Z^2$,
we denote by $\theta_K\I$ the set of trajectories
obtained from~$\I$ by
 removing all trajectories that intersect~$K$
(i.e., we keep only the trajectories that are fully
inside~$\Z^2\setminus K$), 
and by $\sigma_K\I$ the (a.s.\ finite) set of trajectories
that intersect~$K$; clearly, $\I=\theta_K\I\cup \sigma_K\I$
and $\theta_K\I$ is independent of~$\sigma_K\I$.
For finite~$K$, we denote by~$|\sigma_K\I|$
the number of $\I$'s trajectories that intersect~$K$;
note that 
$|\sigma_K\I| 
 \sim \text{Poisson}(\pi\alpha\capa(\{0\}\cup K))$.

Now, what is the correct way to define a tail
event, for random interlacements?
Informally, if one wants to know if such an event occurs,
it is enough to look at the interlacement configuration
outside of any finite set, i.e., it only depends on
``what happens at infinity''. 
The precise definition of tail events
depends on what exactly
is meant by ``interlacement configuration
on a set'' --- one can just keep track of vacant/occupied
sites, or the field of local times, or somehow keep track of
(finite or infinite) pieces of trajectories that belong to that
set (possibly even specifying which of these pieces are parts
of the same infinite trajectory and in which order), etc.
Here we adopt a rather general approach which still permits
us to keep the notations relatively simple.

Let~$A$ be any subset of~$\Z^2\setminus\{0\}$ and let~$I$ be
a (discrete) interval in~$\Z$, finite or infinite
(i.e., $I=(a,b)\cap\Z$, where $a,b\in\R\cup\{\pm\infty\}$).
A \emph{noodle}~$w$ on~$A$ is simply a (finite or infinite) 
nearest-neighbor sequence of sites of~$A$ indexed by some interval~$I$:
$w= (x_k, k\in I)$ such that $x_k\sim x_{k+1}$
when $k,k+1\in I$. Then, the interlacement configuration
on~$A$ is the multiset (i.e., a collection
without ordering but possibly 
with repetitions) of noodles generated by the interlacement
trajectories in a natural way: 
if $\vr=(\vr(k))_{k\in\Z}$ is
a RI's trajectory, then the index set $\{k: \vr(k)\in A\}$ 
is uniquely represented as $I_1\cup I_2\cup\ldots$,
a (finite or infinite) 
union of nonadjacent discrete intervals. The 
noodles generated by that trajectory are 
$w^{(\vr)}_1=(\vr(i), i\in I_1),
w^{(\vr)}_2=(\vr(j), j\in I_2)$, and so on;
we then take the noodles generated by all the trajectories,
and ``store'' them in the multiset.
Note that, unless a noodle is indexed by~$\Z$
(i.e., it is a whole RI's trajectory), its first (last)
site should be a neighbor of 
$\Z^2\setminus A$.
It is also clear that the interlacement configuration
defined in this way is ``informative enough'': having a
multiset of noodles on~$A$, one can figure out which sites
of~$A$ are vacant, calculate the local times, etc.
Also, we will use the following notation:
if~$\I$ is the whole $\RI(\alpha)$ configuration
(i.e., the set of all its trajectories), then
$\I(A)$ will denote the interlacement configuration
(or ``noodle configuration'')
on~$A$.
Note also that $\I(A)=\sigma_A\I(A)$.

Now, let $\F_n$ be the sigma-algebra generated by
the cylinder events 
\begin{align*}
 \lefteqn{
 \big(\{\I(K)\in\NN\}, K \text{ is a finite subset of }  
 \Theta_n,
 }\\
 &\qquad \qquad \qquad \quad
 \NN \text{ is a set of finite noodle 
     configurations on } K \big);
\end{align*}
 define $\Tsp=\bigcap_{n\geq 1}\F_n$,
the sigma-algebra of tail events. The key
result (which implies~\eqref{01_vacant}) is:
\begin{theo}
\label{t_spacewise_01}
 If $E\in\Tsp$, then $\IP[E]=0$ or~$1$.
\end{theo}

\begin{proof}
The proof of this result is similar to the arguments in Section 3.2 of~\cite{ProTyk11}, 
but with necessary adaptations to our situation.
First,
we need to recall an elementary
 result on the total variation distance~$\distTV$
between Poisson distributions of different
(but relatively close) rates, 
as well as between a Poisson distribution and its 
shifted version:
\begin{lem}
\label{l_Poisson_TV}
 There exists a universal constant $c$
 such that, for all $\lambda, h>0$,
\begin{equation}
\label{eq_Poisson_TV}
\distTV\big(\text{Poisson}(\lambda+h),
 \text{Poisson}(\lambda)\big)
\leq c \frac{h}{\sqrt{\lambda}}.
\end{equation}
Also, we have
\begin{equation}
\label{eq_Poisson+1_TV}
\distTV\big(\text{Poisson}(\lambda),
 \text{Poisson}(\lambda)+1\big)
\leq \frac{1}{2\sqrt{\lambda}}.
\end{equation}
\end{lem}
\begin{proof}
 Being $X\sim \text{Poisson}(\lambda)$, write
(using $\IE a^X = e^{\lambda(a-1)}$)
\begin{align*}
2\distTV\big(\text{Poisson}(\lambda+h),
 \text{Poisson}(\lambda)\big)
 &= 
 \sum_{k=0}^\infty
   \Big|e^{-(\lambda+h)}\frac{(\lambda+h)^k}{k!}
   -e^{-\lambda}\frac{\lambda^k}{k!} \Big|\\
  &= \IE \big|e^{-h}(1+\tfrac{h}{\lambda})^X-1\big| \\
& \leq \sqrt{\IE \big(e^{-h}(1+\tfrac{h}{\lambda})^X-1\big)^2}
\\ 
 &= \sqrt{e^{h^2/\lambda}-1}, 
\end{align*}
which implies~\eqref{eq_Poisson_TV}. Similarly,
we have
\[
 2\distTV\big(\text{Poisson}(\lambda),
 \text{Poisson}(\lambda)+1\big)
  = \IE \Big|\frac{X}{\lambda}-1\Big|
  \leq \frac{1}{\lambda}\sqrt{\IE(X-\lambda)^2}
  =\frac{1}{\sqrt{\lambda}},
\]
which shows~\eqref{eq_Poisson+1_TV}.
\end{proof}

Next, we argue that the trajectories that pass close
to the origin are ``not important at infinity'',
in the following sense:
%
\begin{lem}
\label{l_couple_RI}
Fix a finite~$K\subset \Z^2$ and let~$\NN$ be a
set of finite noodle configurations on~$K$.
Then, for any $\eps>0$
there exists a coupling of two copies~$\I$ and~$\J$ 
of~$\RI(\alpha)$
such that
\begin{itemize}
 \item[(i)] $\I$ and $\J(K)$ are independent;
 \item[(ii)] 
 there exists
(large enough)~$n$ such that
  \begin{equation}
\label{eq_couple_RI}
 \IP\big[\I(\Theta_n) = \J(\Theta_n) \mid \J(K)\in\NN\big] 
  \geq 1-\eps.
\end{equation}
\end{itemize}
\end{lem}

\begin{proof}
Let~$\xi_K=|\sigma_K \J|$
be the number of $\J$'s trajectories
that intersect~$K$. We denote these trajectories
by $\vr^{(1)}, \ldots, \vr^{(\xi_K)}$,
where $\vr^{(j)}=(\vr^{(j)}(m), m\in\Z)$.
Let (see Figure~\ref{f_xi_D})
\begin{align*}
 \tau^{(j)}_- &= \min\big\{m: \vr^{(j)}(m)\in K\big\},\\
 \tau^{(j)}_+ &= \max\big\{m: \vr^{(j)}(m)\in K\big\}
\end{align*}
be the times when the $j$th trajectory first enters~$K$
and leaves~$K$ for good; then, 
$(\vr^{(j)}(\tau^{(j)}_- -k), k=0,1,2,\ldots)$
and
$(\vr^{(j)}(\tau^{(j)}_+ +k), k=0,1,2,\ldots)$
are conditioned SRW's trajectories also conditioned
on not re-entering~$K$; we will refer to these
as \emph{escape trajectories}.
An important observation is that the escape trajectories
are conditionally independent from the 
``inner part'' $(\vr^{(j)}(m), j=1,\ldots,\xi_K,
     \tau^{(j)}_-\leq m \leq \tau^{(j)}_+)$:
to obtain the interlacement configuration on~$K$,
it suffices to generate the above collection of 
finite pieces of trajectories; to obtain the whole
trajectories of~$\sigma_K\J$, one has to further
run~$2\xi_K$
independent conditioned (on not hitting $K\cup\{0\}$)
random walks started at 
$(\vr^{(j)}(\tau^{(j)}_\pm), j=1,\dots,\xi_K)$,
but this can be done at a later stage of the coupling's
construction.
     Define (again, see Figure~\ref{f_xi_D})
\[
 D_K = \max_{\substack{j=1,\ldots,\xi_K\\
     \tau^{(j)}_-\leq m \leq \tau^{(j)}_+ }} 
     \|\vr^{(j)}(m)\| 
\] 
to be the maximal distance from the origin achieved by these 
finite pieces of the ``inner part''.
\begin{figure}
\begin{center}
\includegraphics{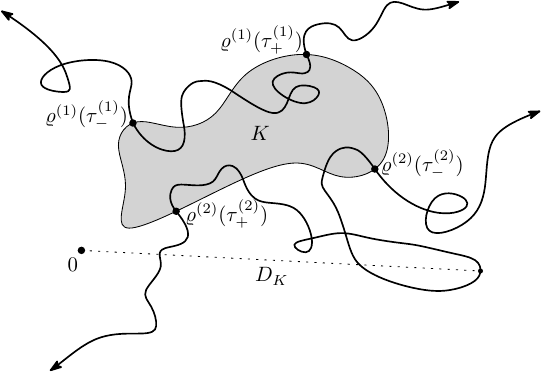}
\caption{On the definition of the auxiliary random variables
(here, $\xi_K=2$).}
\label{f_xi_D}
\end{center}
\end{figure}

For a given~$\eps>0$, let us choose $m_0, \gamma_0$
in such a way that 
\begin{equation}
\label{control_xi_D}
 \IP\big[\xi_K\leq m_0, D_K\leq \gamma_0
    \mid \J(K)\in\NN\big]\geq 1-\frac{\eps}{2}.
\end{equation}

The idea of the proof is illustrated 
on Figure~\ref{f_couplingRI}: 
we keep the trajectories outside~$\Lambda_{\ln n}$
the same in both the interlacement processes~$\I$
and~$\J$, 
but (after obtaining the value of~$\xi_K$) resample those
that intersect~$\Lambda_{\ln n}$ in such a way that,
with high probability 
with respect to $\IP[\;\cdot \mid \J(K)\in\NN]$,
the total numbers of such trajectories in~$\I$ and~$\J$
are equal; then, we couple these trajectories
on their first entrances to~$\Theta_n$ and argue that,
with high probability, these trajectories will remain 
coupled forever. 
\begin{figure}
\begin{center}
\includegraphics{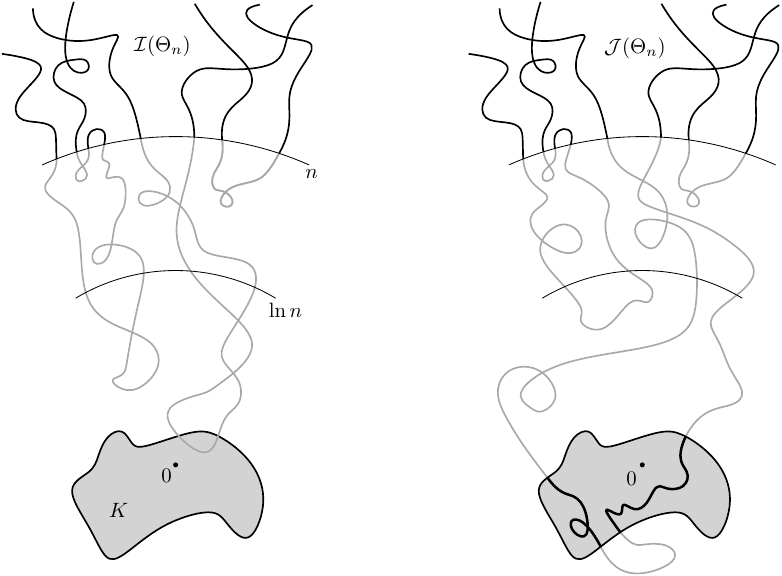}
\caption{Construction of the coupling.}
\label{f_couplingRI}
\end{center}
\end{figure}

We now describe the construction in a more detailed way. 
First, we sample the ``inner part'' (see above)
of the trajectories of~$\sigma_K\J$ in such a way
that the event
\[
\{\J(K)\in\NN, \xi_K\leq m_0, D_K\leq \gamma_0\}
\]
occurs;
this happens with $\IP[\;\cdot \mid \J(K)\in\NN]$-probability
at least~$1-\frac{\eps}{2}$
(so, from now on we can treat~$\xi_K$ as a fixed number
not exceeding~$m_0$).
Let~$n$ be such that $K\subset\B(\ln n)$
and $D_K\leq n-1$.
Then, we need to couple the numbers of particles
that touch~$\Lambda_{\ln n}$ in both processes;
it is Poisson$(\pi\alpha\capa(\B(\ln n))$ in~$\I$
and $\xi_K+\text{Poisson}\big(\pi\alpha
\big(\capa(\B(\ln n))-\capa(\{0\}\cup K)\big)\big)$
in~$\J$.
By Lemma~\ref{l_Poisson_TV} 
(note that $\capa(\B(\ln n))\asymp \ln\ln n$), 
this coupling will be successful
with probability at least 
$1 - O\big(\tfrac{m_0}{\sqrt{\ln\ln n}}\big)$.
So, given that the coupling is a success and
abbreviating $Y=|\sigma_{\Lambda_{\ln n}}\I|$,
we then have
\begin{itemize}
\item in~$\I$, there are~$Y$ 
 (independent) random walks originating somewhere
 at $\Lambda_{\ln n}$ and conditioned on not hitting~$\{0\}$,
 and~$Y$ random walks originating
 at $\Lambda_{\ln n}$ and conditioned 
 on not re-entering~$\Lambda_{\ln n}$;
\item in~$\J$, there are~$2\xi_K$ random walks 
 (which are the ``escape trajectories'' discussed
 above) originating at~$\partial K$ and conditioned
 on not hitting $\{0\}\cup K$, 
 $Y-\xi_K$ random walks originating
 at $\Lambda_{\ln n}$ and conditioned 
 on not hitting~$\{0\}\cup K$, and
 $Y-\xi_K$ random walks originating
 at $\Lambda_{\ln n}$ and conditioned 
 on not re-entering~$\Lambda_{\ln n}$.
\end{itemize}

Next, by~\cite[Proposition~6.4.5]{LL10} (which controls
the conditional exit measure to the boundary of a large
disk and can be used to obtain that all the above random 
walkers have essentially the same entrance measure 
to~$\Theta_n$), 
the probability of successfully coupling
two conditioned random walkers originating in $\B(\ln n)$
to have the common entry point to~$\Theta_n$
is $1 - O\big(\tfrac{\ln n \ln\ln n}{n}\big)$
and we need to take care of only $O(\ln\ln n)$
trajectories. Now, it remains to assure that
with high probability these pairs of trajectories will remain 
coupled (note that they are conditioned on different things).
For this, note that (for a random walk~$\s$ conditioned
on not hitting the origin)
the probability of ever reaching~$\Lambda_{\ln n}$
starting somewhere at~$\Lambda_n$ is 
$O\big(\tfrac{\ln\ln n}{\ln n}\big)$ (by, e.g.,
Lemma~3.4 of~\cite{CPV16}),
and we have to deal with $O(\ln \ln n)$ walkers.
Therefore, on arrival to~$\Theta_n$,
we can just use the same piece of trajectory
(sampled according to the law of the conditioned SRW)
for each pair, and this will be successful
with probability at least
$1-O\big(\tfrac{(\ln\ln n)^2}{\ln n}\big)$.
Gathering the pieces, we obtain that
 the overall probability of success of the coupling
(when the event in~\eqref{control_xi_D}
occurs) is 
at least 
$1 - \tfrac{cm_0}{\sqrt{\ln\ln n}}$,
for a large enough~$c$.
We can then choose~$n$ in such a way that
$\tfrac{cm_0}{\sqrt{\ln\ln n}}\leq \tfrac{\eps}{2}$,
and this concludes the proof of Lemma~\ref{l_couple_RI}.
\end{proof}

Now, we are ready to finish the proof 
of Theorem~\ref{t_spacewise_01}. Let~$E$ be a tail event. 
For a finite~$K\subset\Z^2$
and~$\NN$ a set of noodle configurations on~$K$, consider 
the two copies~$\I$ and~$\J$ of~$\RI(\alpha)$
as in Lemma~\ref{l_couple_RI}; let~$\cCC$
be the event that the coupling is successful.
We can write
\begin{align*}
\lefteqn{
 \IP\big[\I\in E\big] 
}\\
\intertext{\qquad 
\footnotesize (since $E$ is a tail event
and by (i) of Lemma~\ref{l_couple_RI})}
     &= \IP\big[\I(\Theta_n)\in E
     \mid \J(K)\in\NN\big]  \\
 \intertext{\qquad 
 \footnotesize (since, on $\cCC$, we have
$\I(\Theta_n)=\J(\Theta_n)$) }
    &= \IP\big[\J(\Theta_n)\in E, \cCC
     \mid \J(K)\in\NN\big] 
     + \IP\big[\I(\Theta_n)\in E, \cCC^\complement
     \mid \J(K)\in\NN\big] \\
\intertext{\qquad      
\footnotesize (again, since $E$ is a tail event)}
    &= \IP\big[\J\in E
     \mid \J(K)\in\NN\big]
      -\IP\big[\J\in E, \cCC^\complement
     \mid \J(K)\in\NN\big]\\
 & \qquad   {}   + \IP\big[\I\in E, \cCC^\complement
     \mid \J(K)\in\NN\big],
\end{align*}
which means that, since $\J$ is a copy of~$\I$
and by~\eqref{eq_couple_RI}
\[
\Big|\IP\big[\I\in E\big]-\IP\big[\I\in E
     \mid \I(K)\in\NN\big]\Big| 
     \leq \IP[\cCC^\complement \mid \J(K)\in\NN]
     \leq \eps.
\]
Using that~$\eps$
is arbitrary, we obtain that the events
$\{\I\in E\}$ and $\{\I(K)\in\NN\}$ are independent
for any choice of (finite) $K$ and~$\NN$.
In a standard way,\footnote{Observe that all events
independent to a given event form a $\lambda$-system,
and $(\{\I(K)\in\NN\})$ is a $\pi$-system.}
Dynkin's $\pi$-$\lambda$ theorem then implies that
the event $\{\I\in E\}$ is independent of itself,
and so its probability must be equal to~$0$ or~$1$.
 This concludes
the proof of Theorem~\ref{t_spacewise_01}.
\end{proof}

\begin{rem}
It is worth mentioning that all the discussion of this
paper also applies to two-dimensional
 Brownian random interlacements~\cite{CP20}.
In particular, one can define the tail events 
and prove the 0-1 law, and
essentially the same proof
of the fact that the 
$s$-interior of the critical 
vacant set~$\mathcal{D}_s(\V^1)$ is a.s.\ unbounded
works in the continuous setting.
\end{rem}

\section*{Acknowledgements}
The authors thank Y.~Abe and O.~Zeitouni for discussions
on excursions and cover times (as mentioned in the proof
of Proposition~\ref{p_ind_exc_notcover});
also, they thanks the referees for useful comments and
suggestions.
The authors were partially supported by
CMUP, member of LASI, which is financed by national funds
through FCT --- Funda\c{c}\~ao
para a Ci\^encia e a Tecnologia, I.P., 
under the project with reference UIDB/00144/2020.

\end{document}